\documentclass[12pt]{amsart}
\usepackage{a4}
\usepackage[T1]{fontenc}
\usepackage[utf8]{inputenc}
\usepackage[english]{babel}
\usepackage{amsmath}
\usepackage{amsthm}
\usepackage{amssymb}
\sloppypar

\newtheorem{thm}{Theorem}[section]

\newtheorem{lem}[thm]{Lemma}

\newtheorem{cor}[thm]{Corollary}

\theoremstyle{definition}
\newtheorem{de}[thm]{Definition}

\numberwithin{equation}{section}

\newcommand{\N}{\mathbb{N}}

\newcommand{\Som}{\operatorname{Som}}
\usepackage{bm}
 
\newcommand{\ob}[1]{{\mathbb{#1}}} 
\newcommand{\vb}[1]{\bm{#1}} 
\newcommand{\algop}[2]{( {#1}, {#2} )}
\newcommand{\ups}{{\uparrow}}

\newcommand{\VecTwo}[2]{
   \big(
   \begin{smallmatrix}
      #1 \\ #2
   \end{smallmatrix}
   \big)
   }

\date{\today}

\title[Some well partially ordered sets]{Dickson's Lemma, Higman's Theorem and Beyond: a survey of
       some basic results in order theory}

\author{Erhard Aichinger}
\address{
Institut für Algebra,
Johannes Kepler Universit\"at Linz, Altenberger Strasse 69, 4040 Linz,
Austria}
\email{erhard@algebra.uni-linz.ac.at}

\author{Florian Aichinger}
\address{
 Institut für Algebra,
Johannes Kepler Universit\"at Linz, Altenberger Strasse 69, 4040 Linz,
Austria}
\email{florian@algebra.uni-linz.ac.at}

\thanks{
  Supported by the Austrian Science Fund (FWF):~P29931.
       }

\begin{document}

\begin{abstract}
  We provide proofs for the fact that certain orders
  have no descending chains and no antichains.
\end{abstract}    
  
\maketitle

\section{Introduction} \label{sec:intro}

We investigate some finiteness conditions of a partially ordered set $\algop{A}{\le}$.
As usual, we write $a \ge b$ for $b \le a$, and $a < b$ (or $b > a$) for $a \le b$ and $a \neq b$. Furthermore, $a \perp b$ stands for
($a \not\le b$ and $b \not \le a$); in this case, we call $a$ and $b$
\emph{uncomparable}.
A sequence $(a_i)_{i \in \N_0}$ from $A$ is a \emph{descending chain} if
$a_i > a_{i+1}$ for all $i \in \N_0$, it is an \emph{antichain}
if for all $i, j \in \N_0$, $a_i \le a_j$ implies $i = j$, and it is
an \emph{ascending chain} if for all $i \in \N_0$, $a_i < a_{i+1}$.
A subset $U$ of $A$ is called \emph{upward closed} if
$u \in U$, $a \in A$, and $u \le a$ imply $a \in U$.
For a subset $B$ of $A$, we define the \emph{upward closed set generated by $B$}
by $\ups B := \{ a \in A \mid \exists b \in B : b \le a \}$.
By $U \algop{A}{\le}$ or simply $U(A)$ we denote the
set of upward closed subsets of $A$. This set can be ordered by set inclusion $\subseteq$.

One frequently uses the fact that certain partially ordered sets have no descending chain
and no antichain; such orders are called \emph{well partial orders}.
\cite{La:WASO,AH:FTIS} provide powerful techniques to establish that
a given order is a well partial order.
In this note, we restrict our attention to some particular ordered sets:
The first set that we consider is
the set $\N_0^m$ of vectors of natural numbers of some fixed length $m$,
which we order by $(a_1, \ldots, a_m)
\le (b_1, \ldots, b_m)$ if $a_i \le b_i$ for all $i \in \{1,\ldots,m\}$, and
we provide proofs of the following well known facts:
\begin{thm} \label{thm:n}
  Let $m \in \N$.
  \begin{enumerate}
  \item \label{it:n1} $(\N_0^m, \le)$ has no descending chain and no antichain \cite[Lemma~A]{Di:FOTO}.
  \item \label{it:n2} $(U(\N_0^m, \le), \subseteq)$ has no ascending chain and no antichain  \cite{Ma:AOMI}, \cite[Corollary~1.8]{AP:OOMI}.
  \end{enumerate}
\end{thm}
It is easy to see that $(\N_0^m, \le)$ has no descending chain.
In 1913, Dickson proved that $(\N_0^m, \le)$ has no antichain, a fact
that lies at the basis of the theory of Gr\"obner bases \cite{BL:AS,Bu:EAKF}.
This fact can be stated differently. We call an ideal of the
polynomial ring $\mathbb{Q}[x_1, \ldots, x_N]$ \emph{monomial} if it is
generated by monomials. Then Dickson's Lemma states that
every monomial ideal is finitely generated (which of course also
follows from Hilbert's Basis Theorem). A somewhat surprising fact is
that the set of monomial ideals of $\mathbb{Q}[x_1, \ldots, x_n]$ has
no antichain, which has been proved in \cite{Ma:AOMI},
but probably much before in an order theoretic setting: in this setting,
the result states that the set of upward closed subsets of
$(\N_0^m, \le)$ has no antichain. A direct proof is given at the
end of Section~\ref{sec:dickson}. Another proof using an ordering
of words over a finite alphabet that goes back to G.\ Higman
\cite{Hi:OBDI} is given in Section~\ref{sec:other}. This is
the second type of order relations we study:
For a finite alphabet $A$, we say that a word $u \in A^* = \bigcup_{n \in \N_0} A^n$ \emph{embeds into}
a word $v$ if $u$ can be obtained from $v$ by cancelling some letters and in this case write $u\le_e v$.
For example $u = aabbca$ embeds into
$v = \underline{a}b\underline{a}\underline{b}a\underline{b}\underline{c}\underline{a}c$.
It follows from \cite{Hi:OBDI} that for a finite set $A$, $(A^*, \le_e)$ has no antichain.
Also, the upward closed subsets of $(A^*, \le_e)$ have no antichain \cite{Na:OWLS}.
Formally, we define when $x \le_e y$ holds by
recursion on the length of $x$.
First, the empty word $\emptyset$ satisfies $\emptyset \le y$ for
all $y \in A^*$. If $x = a u$ with $a \in A$ and $u \in A^*$,
then $x \le_e y$ if there are words $v,w \in A^*$ such that
$y = vaw$ and $u \le_e w$. Then we have:
\begin{thm} \label{thm:a}
  Let $A$ be a finite set.
  \begin{enumerate}
  \item \label{it:a1} $(A^*, \le_e)$ has no descending chain and no antichain \cite{Hi:OBDI}.
  \item \label{it:a2} $(U(A^*, \le_e), \subseteq)$ has no ascending chain and no antichain.
  \end{enumerate}
\end{thm}
We will give a proof of this theorem in Section~\ref{sec:higman}.
  We also investigate the following ordering of words used in \cite{AMM:OTNO}.
  Let $A$ be a finite set, and let $B := (A \times \{0\}) \cup
  (A \times \{1\})$.
  We define a mapping $\varphi : A^* \to B^*$
  by $\varphi (a_1, \ldots ,a_n) := (b_1, \ldots, b_n)$
  with $b_i := (a_i, 0)$ if $a_i \not\in \{a_1, \ldots, a_{i-1}\}$
  and $b_i := (a_i, 1)$ if $a_i \in \{a_1, \ldots, a_{i-1}\}$.
  For $u = (a_1, \ldots, a_n)$, we use $S(u)$ to denote the
  set of letters that occur in $u$, formally
  $S(u) := \{ a_i \mid i \in \{1,\ldots, n\} \}$. 
  For $u, v \in A^*$, we say that $u \le_E v$ if
  $\varphi (u) \le_e \varphi (v)$ and $S(u) = S(v)$.
\begin{thm} \label{thm:E}
  Let $A$ be a finite set.
  \begin{enumerate}
  \item \label{it:E1} $(A^*, \le_E)$ has no descending chain and no antichain
           \cite[Lemma~3.2]{AMM:OTNO}.
  \item \label{it:E2} $(U(A^*, \le_E), \subseteq)$ has no ascending chain and no antichain.
  \end{enumerate}
\end{thm}
The first listed author has used these results on several occasions:
in \cite{Ai:CMCO}, the absence of antichains in $(A^*, \le_e)$ is
used to establish that every constantive Mal'cev clone on a finite
set is finitely related. In \cite{AM:SOCO}, it is proved that
the set of admissible higher commutator operations on a finite congruence
lattice is well partially ordered by a natural ordering.
The variant of Higman's embedding ordering considered in Theorem~\ref{thm:E} was used in \cite{AMM:OTNO}
to generalize the results from \cite{Ai:CMCO}, resulting
in a proof that every finite algebra with few subpowers is finitely related,
and \cite{AM:FGEC} uses the same ordering to prove that
every subvariety of the variety generated by such an algebra is finitely
generated.
Recently, McDevitt \cite{Mc:ACOR} gave a construction of a large
class of word orderings generalizing $\le_e$ and $\le_E$, and he
established a sufficient criterion \cite[Proposition~47]{Mc:ACOR} for
such orders to be well partial orders, thereby generalizing
Theorems~\ref{thm:a}\eqref{it:a1} and \ref{thm:E}\eqref{it:E1}.

An open question in clone theory is whether
there is a finite set with an antichain of clones containing a Mal'cev operation.
One motivation for proving the absence of antichains is that
this absence often
allows testing certain properties of a structure by considering whether
it
contains finitely many forbidden substructures \cite{RS:GM}.
The aim of this note is to establish the order theoretic results that
are listed above and  
used in \cite{Ai:CMCO, AM:SOCO, AMM:OTNO, AM:FGEC} in a rather direct way.
In particular, we will not resort to the theory of better quasi orderings \cite{La:WASO}.
The note is self-contained, in particular we introduce the well known and
very useful concept of 
\emph{minimal bad sequences} due to \cite{Na:OFT}, although this is done in a similar way at
several other places (cf. \cite{AH:FTIS}).
However, we will not give a proof of the following theorem due to
F.P.\ Ramsey \cite{Ra:OAPO} (cf. \cite{Ne:RT}): Denoting the $2$-element subsets
of $\N_0$ by $\VecTwo{\N_0}{2}$,
Ramsey's Theorem states that for every finite set $T$ and for every
$c : \VecTwo{\N_0}{2} \to T$, there is an infinite subset $Y$ of $\N_0$ such
that the restriction of $c$ to $\VecTwo{Y}{2}$ is constant.

Most of the results and proofs in this note are well known and can
be found, e.g., in \cite{AH:FTIS}. However, we believe that the
the two proofs of Theorem~\ref{thm:n}\eqref{it:n2} and the proof
of Theorem~\ref{thm:a}\eqref{it:a2} are new. The derivation of
Theorem~\ref{thm:E} from Theorem~\ref{thm:a} using quasi-embeddings
(cf. \cite{AH:FTIS}) was suggested to the author by N.\ Ru\v{s}kuc
in 2015 (cf. \cite{Ru:WICA}).
Considering \cite{La:WASO,Na:OWLS,AH:FTIS}, the reader will easily find out
that the theory of well quasi orders and better quasi orders is much
deeper than scratched in the present note. The aim of the present note
is to give easily accessible proofs for some basic and very useful results
in this theory. The first listed author has faced the need
for such proofs when teaching the basics of Gr\"obner Basis Theory, e.g., the existence of universal Gr\"obner bases (cf. \cite{Ai:KA}) or, in universal
algebra \cite{BS:ACIU}, the order theoretical
foundations of the finite relatedness of finite algebras with few
subpowers \cite{AMM:OTNO}. 
\section{Basics of order theory}

It can easily be shown, using, however, some form of the axiom of choice,
that a partially ordered set $\algop{A}{\le}$ has no descending chain
if and only if every finite subset $X$ of $A$ contains a minimal element.
A sequence $(b_i)_{i \in \N_0}$ is a \emph{subsequence} of $(a_i)_{i \in \N_0}$ if
there is a mapping $t : \N_0 \to \N_0$ that satisfies
$i < j \Rightarrow t(i) < t(j)$ for all $i, j \in \N_0$ and
for all $i \in \N_0$, we have $b_i = a_{t(i)}$.  In this case,
$(b_i)_{i \in \N_0} = (a_{t(i)})_{i \in \N_0}$.
\begin{lem} \label{lem:chainsexist}
  Let $\ob{A} = \algop{A}{\le}$ be a partially ordered set,
  and let $S = (a_i)_{i \in \N_0}$ be a sequence from $A$. Then
  $S$ has a subsequence $T = (a_{t(i)})_{i \in \N_0}$ such that one
  of the following conditions holds:
  \begin{enumerate}
  \item $T = (a_{t(i)})_{i \in \N_0}$ is \emph{constant}, which means
    that for all $i, j \in \N_0$, $a_{t(i)} = a_{t(j)}$.
  \item $T = (a_{t(i)})_{i \in \N_0}$ is a \emph{descending chain}, which means
    that for all $i, j \in \N_0$, $i < j \Rightarrow a_{t(i)} > a_{t(j)}$.
   \item $T = (a_{t(i)})_{i \in \N_0}$ is an \emph{ascending chain}, which means
      that for all $i, j \in \N_0$, $i < j \Rightarrow a_{t(i)} < a_{t(j)}$.   
   \item $T = (a_{t(i)})_{i \in \N_0}$ is an \emph{antichain}, which means
     that for all $i, j \in \N_0$, $i < j \Rightarrow a_{t(i)} \perp a_{t(j)}$.
  \end{enumerate}
\end{lem}
\emph{Proof:} We define a coloring $c$ of the $2$-element subsets of $\N_0$.
Let $i, j \in \N_0$ with $i < j$. We set
$c ( \{ i, j\} ) := 1$ if $a_i = a_j$,
$c ( \{ i, j\} ) := 2$ if $a_i < a_j$,
$c ( \{ i, j\} ) := 3$ if $a_i > a_j$, and
$c ( \{ i, j\} ) := 4$ if $a_i \perp a_j$.
By Ramsey's Theorem, $\N$ contains an infinite subset $Y$ such that
$c$ is constant on $\VecTwo{Y}{2}$. We let $t : \N_0 \to Y$ be an injective
increasing function from $\N_0$ into $Y$. Then $T := (a_{t(i)})_{i \in \N_0}$ is
the required subsequence. \qed

\section{Bad sequences}

\begin{de}
  Let $\algop{A}{\le}$ be a partially ordered set.
  A sequence $(a_i)_{i \in \N}$ from $A$ is \emph{good} if there are
  $i, j \in \N_0$ such that $i < j$ and $a_i \le a_j$, and
  it is \emph{bad} if it is not good.
\end{de}

\begin{lem} \label{lem:badseq}
  Let $\ob{A} = \algop{A}{\le}$ be an ordered set. The following are
  equivalent:
  \begin{enumerate}
  \item Every sequence from $A$ is good.
  \item $\ob{A}$ has no descending chain and no antichain.
  \end{enumerate}
\end{lem}
\emph{Proof:}
Assume that every sequence from $A$ is good. Then $\ob{A}$ has no
descending chain and no antichain, since such chains are all bad.
Now assume that $\ob{A}$ has no descending chain and no antichain,
and let $(a_i)_{i \in \N_0}$ be a sequence from $A$. Then by Lemma~\ref{lem:chainsexist},
$(a_i)_{i \in \N_0}$ has a subsequence $T = (a_{t(i)})_{i \in \N_0}$ that is constant, a descending chain,
an ascending chain, or an antichain. Descending chains and antichains are
excluded by the assumptions, thus $T$ is either constant or ascending.
In both cases, $a_{t(1)} \le a_{t(2)}$, and hence $T$ is good. \qed

We call a sequence $(a_i)_{i \in \N_0}$ a \emph{minimal bad sequence}
if it is a bad sequence, and for every $i \in \N_0$
and for every $b < a_i$, every sequence starting with
$(a_0, a_1, \ldots, a_{i-1}, b)$ is good. 

\begin{lem} \label{lem:minbad}
  Let $\ob{A} = \algop{A}{\le}$ be an ordered set.
  If $\ob{A}$ has no descending chain and it has an infinite antichain,
  then it has a minimal bad sequence.
\end{lem}
 \emph{Proof:}
We assume that $(A, \le)$ has an antichain. This antichain is
  a bad sequence.
  Inductively, we can define a minimal bad sequence.
  We define $a_0$ to be a minimal element with respect to $\le$ of
  $S_0 := \{ y_0 \mid (y_i)_{i \in \N_0} \text{ is a bad sequence from $A$} \}$.
  For $j \in \N$, having defined $(a_0, \ldots, a_{j-1})$, we let
  \[ 
      S_j := \{ y_j \mid (y_i)_{i \in \N_0} \text{ is a bad sequence from $A$ with }
  (y_0, \ldots, y_{j-1}) = (a_0, \ldots, a_{j-1}) \},
   \]
  and we choose
  $a_j$ to be a minimal element of $S_j$ with respect to $\le$. \qed

  \begin{lem} \label{lem:acc}
  Let $\algop{A}{\le}$ be a partially ordered set.
  The following are equivalent:
  \begin{enumerate}
\item $\algop{A}{\le}$ has no descending chain and no antichain.
\item $\algop{U(A, \le)}{\subseteq}$ has no ascending chain.
   \end{enumerate} 
\end{lem}
\emph{Proof:}
Let $(a_i)_{i \in \N}$ be a descending chain or an antichain from $A$,
and let $B_i := \ups \{a_1,\ldots, a_i\}$ for $i \in \N_0$. 
Then $(B_i)_{i \in \N}$ is an ascending chain.
For the other implication, let $(C_i)_{i \in \N_0}$ be an ascending chain,
and choose $c_i \in C_{i+1} \setminus C_i$. We show that $(c_i)_{i \in \N_0}$
  is bad. Suppose $i < j$ and $c_i \le c_j$. Since $C_{i+1}$ is
  upward closed, we then have $c_j \in C_{i+1}$, and thus
  $c_j \in C_j$. This contradicts the choice of $c_j$ in $C_{j+1} \setminus C_j$.
  Hence $\algop{A}{\le}$ has a bad sequence, and thus by Lemma~\ref{lem:badseq},
  it must have a descending chain or an antichain.
\qed
  
\section{Dicksons's ordering} \label{sec:dickson}

In this section we provide a direct proof for Theorem~\ref{thm:n} by using Ramsey's theorem. In Section~\ref{sec:other} we show that these results can also be obtained as a consequence of Theorem~\ref{thm:a} by using quasi-embeddings. 

\begin{proof}[Proof of {Theorem~\ref{thm:n}\eqref{it:a1}}.]
For  $m \in \N$ and $k \in \{1,\ldots,m\}$ we denote the $m$-tuple $(a_1,\dots,a_m)\in \N_0^m$ by $\vb{a}$ and the
   $k$\,th component of 
   $\vb{a}$ by $\vb{a}_k$. 
   
   It can be easily seen that $(\N_0^m,\le)$ has no descending chain.
   
   We assume that $(\N_0^m,\le)$ has an antichain $(\vb{a}^{(i)})_{i \in \N_0 }$. Now we color the $2$-element subsets of  $\N_0$
   with  the elements of ${\{1,2\}}^{\{1,\dots, m\}}$ as colors.  
   For $i<j$, we set
   \[ 
      c (\{i,j\}) \, (k) := \left\{
        \begin{array}{rl}
            1 & \text{ if } \vb{a}^{(i)}_k \le \vb{a}^{(j)}_k, \\
            2   & \text{ if } \vb{a}^{(i)}_k >   \vb{a}^{(j)}_k.
        \end{array}
        \right.   
   \]
   By Ramsey's Theorem, we find a subsequence $(\vb{a}^{t(i)})_{i \in \N_0}$
    and a color $C \in \{1,2\}^{\{1,\dots, m\}}$ such that for all $i,j \in \N_0$ with $i <j$,
  we have
  $c ( \{t(i), t(j)\} ) = C$. 

   Assume there exists $k\in {\{1,\dots, m\}}$ such that 
   $C(k) = \,\, 2$. 
   Then $\vb{a}^{(t_0)}_k > \vb{a}^{(t_1)}_k > \vb{a}^{(t_2)}_k > \cdots$ in contradiction to the fact that $(\N_0, \le)$ has no descending chain. 
   
   Hence $C(k) = 1$ for all $k$ and therefore  
    $\vb{a}^{(t_0)} \leq \vb{a}^{(t_1)} \leq \vb{a}^{(t_2)} 
     \leq \cdots$, contradicting our assumption of $(\vb{a}^{(i)})_{i \in \N_0 }$ being an antichain.
\end{proof}

\begin{proof}[Proof of {Theorem~\ref{thm:n}\eqref{it:a2}}.]
By Theorem~\ref{thm:n}\eqref{it:a1}, $(\N_0^m, \le)$ has
  no descending chain and no antichain. Hence Lemma~\ref{lem:acc}
  yields that $(U(\N_0^m, \le), \subseteq)$ has no ascending chain.
  It remains to show that $(U(\N_0^m, \le), \subseteq)$ has no antichain.
If $m = 1$, the set of upward closed subsets of $\N_0$ is totally ordered, and hence there cannot be an antichain. In the case
    $m \ge 2$, we encode each upward closed subset of $\N_0^m$ by
a function from $\N_0^{m-1}$ to $\N_0$. 
    For each upward closed subset $F$ of $\N_0^m$ we define a function
    $\Phi_F : \N_0^{m-1} \to \N_0 \cup \{ \infty \}$ by
    \[
       \Phi_F (\vb{a}) := 
          \left\{
             \begin{array}{rl}
                 \min \{ c \in \N_0 |%
                         (\vb{a},c) \in F \}
                        & \text{ if there exists $c' \in \N$ such that
                                  $(\vb{a},c') \in F$ }, \\
                  \infty & \text{ otherwise.}
             \end{array}
          \right.
    \]
    for $\vb{a} \in \N_0^{m-1}$.
    First we show that for all  $\vb{a}, \vb{b} \in \N_0^{m-1}$
    such that $\vb{a} \le \vb{b}$, also $\Phi_F (\vb{a}) \ge \Phi_F (\vb{b})$ holds.
    Therefore, let $c := \Phi_F (\vb{a})$. We assume $c \neq \infty$.
    We have $(\vb{a}, c) \in F$. Since $F$ is an upward closed set, also $(\vb{b}, c) \in F$, and hence
    $\Phi_F (\vb{b}) \le c = \Phi_F (\vb{a})$.
    Moreover, for the upward closed subsets $F,G$ of $\N_0^m$ the inclusion
    $F \subseteq G$ holds if and only if
    $\Phi_F (\vb{a}) \ge \Phi_G (\vb{a})$ for all
    $\vb{a} \in \N_0^{m-1}$.
    
    Let  $(F_i)_{i \in \N_0}$ be an antichain in $(U(\N_0^m, \le), \subseteq)$.
    Thus for $i, j \in \N$ with $i < j$, 
    $F_j \not\subseteq F_i$ holds. Accordingly, there exists
    $\vb{a}^{(i,j)} \in \N_0^{m-1}$ such that
    \[
       \Phi_{F_j} (\vb{a}^{(i,j)}) < \Phi_{F_i}(\vb{a}^{(i,j)}).
    \]
    Now we color the $3$-element subsets of  $\N_0$
   with  the elements of ${\{1,2\}}^{\{1,\dots, m-1\}}$ as colors.  
   For $l \in \{1,\ldots,m-1\}$, we denote the $l$\,th component of
  $\vb{a}^{(i,j)}$ by $\vb{a}^{(i,j)}_l$.
  For $i<j<k$ we define the coloring of $\{i,j,k\}$ in the following way:
   \[ 
      c (\{i,j, k\}) \, (l) := \left\{
        \begin{array}{rl}
            1 & \text{if } \vb{a}^{(i,j)}_l \le \vb{a}^{(j,k)}_l, \\
            2   & \text{if } \vb{a}^{(i,j)}_l  >  \vb{a}^{(j,k)}_l.
        \end{array}
        \right.   
   \]
   By Ramsey's Theorem, we find an infinite subset $T$ of $\N_0$, $T=\{t_1, t_2,\dots\}$, $t_1<t_2<\dots$,
    and a color $C \in \{1,2\}^{\{1,\dots, m-1\}}$ such that for all $i,j,k \in \N_0$ with $i<j<k$,
  we have
  $c ( \{t_i, t_j, t_k\} ) = C$. 
   We now show that $C(l) = \,\, 1$
   for all $l \in \{1,\ldots,m-1\}$.
    
  In contradiction to that, we assume that there exists $l$ such that 
  $C(l)\, = \,\,2$.
  Then
   \[
       \vb{a}^{(t_0, t_1)}_l > \vb{a}^{(t_1, t_2)}_l > \vb{a}^{(t_2,t_3)}_l> \cdots.
   \]  holds.
  Hence we have constructed a descending chain of natural numbers, which is impossible. 
   
  Thus for all $r \in \N_0$ the inequality
  $\vb{a}^{(t_r, t_{r+1})} \le \vb{a}^{(t_{r+1}, t_{r+2})}$ holds.
  Now let $r \in \N_0$. Because of the choice of
  $\vb{a}^{(t_r, t_{r+1})}$, we have
   \[
      \Phi_{F_{t_r}} (\vb{a}^{(t_r, t_{r+1})})
      >
      \Phi_{F_{t_{r+1}}} (\vb{a}^{(t_r, t_{r+1})}).
   \]
  Since $\vb{a}^{(t_r, t_{r+1})} \le \vb{a}^{(t_{r+1}, t_{r+2})}$,
  we also have
   \[
     \Phi_{F_{t_{r+1}}} (\vb{a}^{(t_r, t_{r+1})})
     \ge 
     \Phi_{F_{t_{r+1}}} (\vb{a}^{(t_{r+1}, t_{r+2})}).
   \]
   Hence the sequence $(\Phi_{F_{t_i}} (\vb{a}^{(t_i, t_{i+1})}))_{i \in \N_0}$ is a descending chain in
  $\N_0 \cup \{\infty\}$, which is impossible.

  Consequently there cannot exist an antichain of upward closed subsets of
   $\N_0^m$. 
\end{proof}

\section{Higman's ordering}  \label{sec:higman}

Of the results listed in Section~\ref{sec:intro}, we now derive
Theorem~\ref{thm:a}.

\begin{proof}[Proof of {Theorem~\ref{thm:a}\eqref{it:a1}}.]
Since $u <_e v$ implies that the length of $u$ is strictly
smaller than the length of $v$, $(A^*, \le_e)$ has no
descending chain.
We assume that $(A^*, \le_e)$ has an antichain.
Then by Lemma~\ref{lem:minbad},
we find
a minimal bad sequence
$U = (u_i)_{i \in \N_0}$ be in $A^*$
Such a sequence cannot contain the empty word, since
$u_i = \emptyset$ implies $u_i \le_e u_{i+1}$, and therefore
$U$ is good.
Hence we can write $u_i = a_i v_i$ with $a_i \in A$ and $v_i \in A^*$.
Since $A$ is finite, there is a subsequence $(a_{t(i)})_{i \in \N_0}$ and $b \in A$
such that $a_{t(i)} = b$ for all $i \in \N_0$.
The sequence $(u_1, u_2, \ldots, u_{t(0)-1}, v_{t(0)}, v_{t(1)}, \ldots)$
is smaller, hence good.
  Therefore, we find $i, j \in \N_0$
  such that either
  $i < j < t(0)$ and $u_i \le_e u_j$, or
  $i < t(0)$ and $u_i \le_e v_{t(j)}$, or
  $i < j$ and $v_{t(i)} \le_e v_{t(j)}$.
In the case  
$i < j < t(0)$ and $u_i \le_e u_j$,
$U$ is good, contradicting the assumptions.
If   $i < t(0)$ and $u_i \le_e v_{t(j)}$, then $u_i \le_e v_{t(j)} \le_e a_j v_{t(j)} = u_{t(j)}$,
and thus $U$ is good, contradicting the assumptions.
If  $i < j$ and $v_{t(i)} \le_e v_{t(j)}$, then
$u_{t(i)} = a_{(t(i))} v_{t(i)} = b v_{t(i)} \le_e
b v_{t(j)}  = a_{(t(j))} v_{t(j)} = u_{t(j)}$, and thus
$U$ is good, again contradicting the assumptions. 
\end{proof}

Before proving Theorem~\ref{thm:a}\eqref{it:a2},
we need some preparation.
For $X \subseteq A^*$ and $a \in A$, we define
$a^{-1}X$ by
\[
a^{-1} X := \{ y \in A^* \mid ay \in X \}.
\]
The set of \emph{starting letters of minimal elements} of $X$
is defined by
\[
\Som (X) := \{ a \in A \mid \exists u \in A^* : au \text{ is a minimal
  element of } X
            \}.
\]
\begin{lem} \label{lem:XaX}
  Let $X, Y$ be a upward closed subsets of $\algop{A^*}{\le_e}$, and let $a \in X$.
  Then we have:
  \begin{enumerate}
  \item \label{it:X1}  $X \subseteq a^{-1} X$.
  \item \label{it:X2} If $a \in \Som (X)$, then $X \subset a^{-1} X$.
  \item \label{it:X3} If $X \neq A^*$ and for all $b \in \Som (X)$,
    $b^{-1} X \subseteq b^{-1} Y$. Then $X \subseteq Y$.
  \end{enumerate}
\end{lem}
\begin{proof}
For item~\eqref{it:X1}, let $x \in X$. Then $x \le_e ax$, and
therefore, since $X$ is upward closed, $ax \in X$. Thus $x \in a^{-1}X$.

For item~\eqref{it:X2}, we choose $y \in A^*$ such that $ay$ is
a minimal element of $X$ with respect to $\le_e$.
Then $y \in a^{-1}X$. Since
$y <_e ay$ and $ay$ is minimal in $X$, we have $y \not\in X$.
Thus the inclusion $X \subset a^{-1} X$ is indeed proper.

For item~\eqref{it:X3}, we fix $x \in X$. Then there is a minimal
element $y$ of $X$ with respect to $\le_e$ such that $y \le_e x$.
In the case that $y$ is the empty word, $X = A^*$, which is excluded
by the assumptions.
If $y$ is not empty, there is $b \in A$ and $z \in A^*$ such that
$y = bz$. Then $b \in \Som (X)$ and $z \in b^{-1} X$.
Therefore $z \in b^{-1} Y$, and therefore $bz \in Y$. Thus $y \in Y$,
and since $Y$ is upward closed, $x \in Y$.
\end{proof}

\begin{proof}[Proof of {Theorem~\ref{thm:a}\eqref{it:a2}}.]
  By Theorem~\ref{thm:a}\eqref{it:a1}, $(A^*, \le)$ has
  no descending chain and no antichain. Hence Lemma~\ref{lem:acc}
  yields that $(U(A^*, \le_e), \subseteq)$ has no ascending chain.
  We call a sequence $(X_i)_{i \in \N_0}$ \emph{co-good} if
  there are $i, j \in \N$ with $i < j$ and $X_i \supseteq X_j$, and
  \emph{co-bad} otherwise.
  We assume that $U(A^*)$ has an antichain. This antichain is
  a co-bad sequence.
  Inductively, we can define a maximal co-bad sequence.
  We define $X_0$ to be a maximal element with respect to $\subseteq$ of
  $S_0 := \{ Y_0 \mid (Y_i)_{i \in \N_0} \text{ is a co-bad sequence of $U(A^*)$} \}$.
  This maximal element exists because $(U(A^*, \le_e), \subseteq)$ has no ascending chain.
  For $j \in \N$, having defined $(X_0, \ldots, X_{j-1})$, we let
  $S_j := \{ Y_j \mid (Y_i)_{i \in \N_0} \text{ is a co-bad sequence of $U(A^*)$ with }
  (Y_0, \ldots, Y_{j-1}) = (X_0, \ldots, X_{j-1}) \}$, and we choose
  $X_j$ to be a maximal element of $S_j$ with respect to $\subseteq$. 
  
  Since the sequence $(\Som (X_i))_{i \in \N_0}$ can take at most $2^{|A|}$
  values, there is $B \subseteq A$ and a subsequence
  $(X_{t(i)})_{i \in \N_0}$ such that $\Som (X_{t(i)}) = B$ for all $i \in \N_0$.
  Now we color the two element subsets of $\N_0$ with the elements
  of $\{1,2\}^B$ as colors. For $i < j$ and $b \in B$, we set
  $c( \{i, j \} ) \, (b) := 1$ if $b^{-1} X_{i} \not\supseteq b^{-1} X_{j}$,
  $c( \{i, j \} ) \, (b) := 2$ if $b^{-1} X_{i} \supseteq b^{-1} X_{j}$.
  We restrict our coloring to the two element subsets
  of $t[\N_0]$. 
  By Ramsey's Theorem, we find a subsequence $(X_{t(r(i))})_{i \in \N_0} =: (X_{s(i)})_{i \in \N_0}$
    and a color $C \in \{1,2\}^B$ such that for all $i,j \in \N_0$ with $i <j$,
  we have
  $c ( \{s(i), s(j)\} ) = C$. Furthermore, $\Som(X_{s(i)}) = B$ for
  all $i \in \N_0$.

  Let us first consider the case that
  we have  $b \in B$ such that
  $C(b) = 1$. To this end, we consider the sequence
  $Y := (X_1, X_2, \ldots, X_{s(0)-1}, b^{-1} X_{s(0)}, b^{-1} X_{s(1)}, \ldots)$.
  Since $b \in \Som (X_{s(0)})$, Lemma~\ref{lem:XaX}\eqref{it:X2} yields
  $X_{s(0)} \subset b^{-1} X_{s(0)}$. By the maximality of
  $(X_{i})_{i \in \N_0}$, the sequence $Y$ is co-good.
  Therefore, we find $i, j \in \N_0$
  such that either
  $i < j < s(0)$ and $X_i \supseteq X_j$, or
  $i < s(0)$ and $X_i \supseteq b^{-1} X_{s(j)}$, or
  $i < j$ and $b^{-1} X_{s(i)} \supseteq b^{-1} X_{s(j)}$.
  The case $i < j < s(0)$ and $X_i \supseteq X_j$
  cannot occur because
  the sequence $(X_i)_{i \in \N_0}$ is co-bad.
    In the case  $i < s(0)$ and  $X_i \supseteq b^{-1} X_{s(j)}$,
    Lemma~\ref{lem:XaX}~\eqref{it:X1} yields
    $X_i \supseteq b^{-1} X_{s(j)} \supseteq X_{s(j)}$. Again this cannot occur because
    $(X_i)_{i \in \N_0}$ is co-bad.
    In the case $i < j$ and  $b^{-1} X_{s(i)} \supseteq b^{-1} X_{s(j)}$, we obtain
    $c (\{s(i), s(j) \}) \, (b) = 2$, contradicting the assumption $C(b) = 1$.  
  
    Hence we have $C(b) = 2$ for all $b \in B$.
    Therefore, for all $b \in \Som (X_{s(1)})$, we have $b^{-1}  (X_{s(0)})
   \supseteq b^{-1}  (X_{s(1)})$. Since $(X_i)_{i \in \N_0}$ is co-bad,
   we have $X_{s(1)} \not\supseteq X_{s(2)}$. From this, we conclude
     $X_{s(1)} \neq A^*$.
    Now Lemma~\ref{lem:XaX}~\eqref{it:X3}
    yields  $X_{s(0)} \supseteq X_{s(1)}$, contradicting the fact that
    $(X_i)_{i \in \N_0}$ is co-bad.
     \end{proof}

\section{Other well partially ordered sets} \label{sec:other}
We will derive the other results stated in Section~\ref{sec:intro}
by using \emph{quasi-embeddings}:
\begin{de} \cite{AH:FTIS}
  Let $\ob{A} = (A, \le_A)$ and $\ob {B} = (B, \le_B)$ be partially ordered sets.
  A mapping $f : A \to B$ is a \emph{quasi-embedding} from $\ob{A}$
  into $\ob{B}$ if
  for all $a_1, a_2 \in A :
  f (a_1) \le_B f (a_2) \Rightarrow a_1 \le_A a_2$.
\end{de}

\begin{lem}   \label{lem:qe1} \cite{AH:FTIS}
  Let $\algop{A}{\le_A}$ and $\algop{B}{\le_B}$ be partially ordered sets.
  Let $f : A \to B$ be a quasi-embedding from
  $\algop{A}{\le_A}$ into $\algop{B}{\le_B}$.
  Then we have:
  \begin{enumerate}
  \item \label{it:a1} If $\algop{A}{\le_A}$ has an antichain, then
    $\algop{B}{\le_B}$ has an antichain.
  \item \label{it:a3} If $\algop{A}{\le_A}$ has a descending chain, then
    $\algop{B}{\le_B}$ has an antichain or a descending chain.
  \end{enumerate}
\end{lem}
\emph{Proof:}
For~\eqref{it:a1}, we let $(a_i)_{i \in \N_0}$ be
an antichain. We define $b_i := f(a_i)$ and we claim
that $(b_i)_{i \in \N}$ is an antichain.
To this end, let $j, k \in \N_0$ with $b_j \le b_k$.
Then $f(a_j) \le f(a_k)$, and therefore $a_j \le a_k$.
Since $(a_i)_{i \in \N_0}$ is an antichain, $j = k$, completing
the proof that $(b_i)_{i \in \N_0}$ is an antichain.

For~\eqref{it:a3},  let $(a_i)_{i \in \N_0}$ be a descending chain.
According to Lemma~\ref{lem:chainsexist}, it is sufficient to show that
that the sequence $(b_i)_{i \in \N_0} := (f(a_i))_{i \in \N_0}$
has no subsequence that is constant or an ascending chain.
Seeking a contradiction, we let $(b_{t(i)})_{i \in \N_0}$ be a subsequence
with $b_{t(1)} \le b_{t(2)} \le \cdots$. Since $f$ is a
quasi-embedding, $a_{t(1)} \le_A a_{t(2)}$, which contradicts the fact that
$(a_i)_{i \in \N_0}$ is descending.
\qed

\begin{lem} \label{lem:qeu}
  Let $\algop{A}{\le_A}$ and $\algop{B}{\le_B}$ be partially ordered sets.
  Let $f : A \to B$ be a quasi-embedding from
  $\algop{A}{\le_A}$ into $\algop{B}{\le_B}$.
  Let
  $\phi : U(A, \le_A) \to U(B, \le_B)$ be defined by
  $\phi (X) := \ups f[X] = \{ b \in B \mid
  \exists x \in X : b \ge_B f(x) \}$.
  Then $\phi$ is a quasi-embedding from
  $(U(A, \le_A), \subseteq)$ into
  $(U(B, \le_B), \subseteq)$.
\end{lem}

\begin{proof}  
We have to show that for all $X, Y \in U(A)$, we have
\begin{equation} \label{eq:x1}
  \ups f[X] \subseteq \ups f[Y] \Rightarrow X \subseteq Y.
\end{equation}
To this end, we fix $X, Y \in U(A)$ and assume
$\ups f[X] \subseteq \ups f[Y]$. We let $x \in X$.
Then $f(x) \in f[X]$. Hence $f(x) \in \ups f[X]$, and therefore
$f(x) \in \ups f[Y]$. Therefore, there exists $y \in Y$ such that
$f(x) \ge_B f(y)$. By the assumptions on $f$, $x \ge_A y$.
Since $Y$ is upward closed, we obtain $x \in Y$. This completes the
proof of~\eqref{eq:x1}.
\end{proof}

\begin{cor} \label{cor:emb}
Let $(A, \le_A)$ and $(B, \le_B)$ be partially ordered sets, and let $f$ be a
quasi-embedding from $(A, \le_A)$ into $(B, \le_B)$.
If $(B, \le_B)$ has no descending chain and no antichain and $(U(B, \le_B), \subseteq)$
has no antichain, then
$(A, \le_A)$ has no descending chain and no antichain, and
$(U(A, \le_A), \subseteq)$ has no ascending chain and no antichain.
\end{cor}
\begin{proof}
  From Lemma~\ref{lem:qe1}\eqref{it:a1}, we obtain that $\algop{A}{\le_A}$
  has no antichain. %
  From~Lemma~\ref{lem:qe1}\eqref{it:a3}, we obtain that
  $\algop{A}{\le_A}$ has no descending chain.
  Lemma~\ref{lem:acc} now yields that $(U(A, \le_A), \subseteq)$ has no
  ascending chain.
  By Lemma~\ref{lem:qeu}, there exists a quasi-embedding from
  $(U(A, \le_A), \subseteq)$ into $(U(B, \le_B), \subseteq)$.
  Applying Lemma~\ref{lem:qe1}\eqref{it:a1} to this quasi-embedding, we obtain
  that $(U(A, \le_A), \subseteq)$ has no antichain.
\end{proof}  
  
\begin{proof}[Proof of {Theorem~\ref{thm:n}}.]
    We first construct
  a quasi-embedding from 
  $\ob{A} := (\N_0^m, \le)$ into $\ob{B} = (\{1,\ldots,m\}^*, \le_e)$.
  We define $f:A \to B$ by
 \[
  f (x_1, \ldots,x_m) := \underbrace{11\ldots 1}_{x_1}
                           \underbrace{22\ldots 2}_{x_2}
                           \ldots
                           \underbrace{mm\ldots m}_{x_m}
   = 1^{x_1}2^{x_2} \ldots m^{x_m}
 \]  
 for all $(x_1, \ldots, x_m) \in \N_o^m$.
  Now if $1^{x_1}2^{x_2} \ldots m^{x_m} \le_e
  1^{y_1}2^{y_2} \ldots m^{y_m}$, then
  $(x_1, \ldots, x_m) \le (y_1, \ldots, y_m)$.
 Thus $f$ is a quasi-embedding. By Corollary~\ref{cor:emb} and Theorem~\ref{thm:a},
 $\ob{A}$ has no antichain and no descending chain and $(U(\ob{A}), \subseteq)$
 has no ascending chain and no antichain.
\end{proof}

\begin{proof}[Proof of Theorem~\ref{thm:E}.]
    We define a mapping $f: A^* \to B^* \times \mathcal{P} (A)$ by
    $f (u) := (\varphi (u), S(u))$.
    We order the set $B^* \times \mathcal{P}(A)$ by
    $(u, S) \le (v, T)$ if $u \le_e v$ and $S = T$.
    Then $f$ is a quasi-embedding from $(A^*, \le_E)$ into 
    $B^* \times \mathcal{P} (A)$.
    Since $(B^*, \le_e)$ has no antichain and $\mathcal{P}(A)$
    is finite, $B^* \times \mathcal{P} (A)$ has no antichain.
    Hence Lemma~\ref{lem:qe1}\eqref{it:a1} implies
    that $(A^*, \le_E)$ has no antichain. Since $u <_E v$ implies that the length of $u$ is strictly smaller than the length of $v$, $(A^*, \le_E)$ has no descending chain. Thus, using Lemma~\ref{lem:acc},
    we get that  $(U(A^*, \le_E), \subseteq)$ has no ascending chain.

    It remains to show that $(U(A^*, \le_E), \subseteq)$ has no 
    antichain. Using Corollary~\ref{cor:emb}, it is sufficient to
    show that $C := ( U(B^* \times \mathcal{P} (A), \le), \subseteq )$ has
    no antichain. Seeking a contradiction, we let
    $(M_i)_{i \in \N_0}$ be an antichain in $C$, and for $S \in \mathcal{P} (A)$,
    let 
    $X(M_i, S) := \{ w \in B^* \mid  (w, S) \in M_i \}$.
    Then $( \langle X(M_i, S) \mid S \in \mathcal{P}(A) \rangle )_{i \in \N_0}$
    is a sequence in $(U (B^*, \le_e))^{\mathcal{P(A)}}$. Since 
    $\mathcal{P}(A)$ is finite, and since $U(B^*, \le_e)$ has no ascending chain
    and no antichain
    by Theorem~\ref{thm:a}\eqref{it:a2}, a modification of the argument
     in the proof of Theorem~\ref{thm:n}\eqref{it:n1} can be used to show that 
    $(U (B^*, \le_e))^{\mathcal{P(A)}}$ has no ascending chain and no antichain.
    Hence by the dual of  Lemma~\ref{lem:badseq}, there are $i, j \in \N$ with
    $i < j$ such that $X(M_i, S) \supseteq X(M_j, S)$ for all $S \in \mathcal{P} (A)$.
    Therefore $M_i \supseteq M_j$, contradicting the fact that $(M_i)_{i \in \N_0}$ 
    is an antichain.
  \end{proof}

\section{Acknowledgements}
   The authors would like to thank J.\ Farley and N.\ Ru\v{s}kuc for numerous
   discussions on the theory of ordered sets. Part of this research was done 
   while the first listed author was visiting P.\ Aglian\`{o} at
   the University of Siena, Italy.
   
\bibliographystyle{alpha}

\begin{thebibliography}{AMM14}

\bibitem[AH07]{AH:FTIS}
M.~Aschenbrenner and R.~Hemmecke.
\newblock Finiteness theorems in stochastic integer programming.
\newblock {\em Found. Comput. Math.}, 7(2):183--227, 2007.

\bibitem[Aic09]{Ai:KA}
E.~Aichinger.
\newblock {K}ommutative {A}lgebra und {A}lgebraische {G}eometrie.
\newblock Lecture notes for a course at JKU Linz, Austria. Available at {\tt
  http://www.algebra.uni-linz.ac.at/Students/KommutativeAlgebra/s09/}, June
  2009.

\bibitem[Aic10]{Ai:CMCO}
E.~Aichinger.
\newblock Constantive {M}al'cev clones on finite sets are finitely related.
\newblock {\em Proc. Amer. Math. Soc.}, 138(10):3501--3507, 2010.

\bibitem[AM13]{AM:SOCO}
E.~Aichinger and N.~Mudrinski.
\newblock Sequences of commutator operations.
\newblock {\em Order}, 30(3):859--867, 2013.

\bibitem[AM16]{AM:FGEC}
E.~Aichinger and P.~Mayr.
\newblock Finitely generated equational classes.
\newblock {\em J. Pure Appl. Algebra}, 220(8):2816--2827, 2016.

\bibitem[AMM14]{AMM:OTNO}
E.~Aichinger, P.~Mayr, and R.~McKenzie.
\newblock On the number of finite algebraic structures.
\newblock {\em J. Eur. Math. Soc. (JEMS)}, 16(8):1673--1686, 2014.

\bibitem[AP04]{AP:OOMI}
M.~Aschenbrenner and W.~Y. Pong.
\newblock Orderings of monomial ideals.
\newblock {\em Fund. Math.}, 181(1):27--74, 2004.

\bibitem[BL83]{BL:AS}
B.~Buchberger and R.~Loos.
\newblock Algebraic simplification.
\newblock In {\em Computer algebra}, pages 11--43. Springer, Vienna, 1983.

\bibitem[BS81]{BS:ACIU}
S.~Burris and H.~P. Sankappanavar.
\newblock {\em A course in universal algebra}.
\newblock Springer New York Heidelberg Berlin, 1981.

\bibitem[Buc70]{Bu:EAKF}
B.~Buchberger.
\newblock Ein algorithmisches {K}riterium f\"ur die {L}\"osbarkeit eines
  algebraischen {G}leichungssystems.
\newblock {\em Aequationes Math.}, 4:374--383, 1970.

\bibitem[Dic13]{Di:FOTO}
L.~E. Dickson.
\newblock Finiteness of the odd perfect and primitive abundant numbers with $n$
  distinct prime factors.
\newblock {\em American Journal of Mathematics}, 35(4):413--422, 1913.

\bibitem[Hig52]{Hi:OBDI}
G.~Higman.
\newblock Ordering by divisibility in abstract algebras.
\newblock {\em Proc. London Math. Soc. (3)}, 2:326--336, 1952.

\bibitem[Lav76]{La:WASO}
R.~Laver.
\newblock Well-quasi-orderings and sets of finite sequences.
\newblock {\em Math. Proc. Cambridge Philos. Soc.}, 79(1):1--10, 1976.

\bibitem[Mac01]{Ma:AOMI}
D.~Maclagan.
\newblock Antichains of monomial ideals are finite.
\newblock {\em Proc. Amer. Math. Soc.}, 129(6):1609--1615 (electronic), 2001.

\bibitem[McD18]{Mc:ACOR}
M.~McDevitt.
\newblock A class of rational relations generalizing the subword order.
\newblock {\em Journal of Automata, Languages and Combinatorics},
  23(4):361--386, 2018.

\bibitem[Ne{\v{s}}95]{Ne:RT}
Jaroslav Ne{\v{s}}et{\v{r}}il.
\newblock Ramsey theory.
\newblock In {\em Handbook of combinatorics, Vol.~2}, pages 1331--1403.
  Elsevier, Amsterdam, 1995.

\bibitem[NW63]{Na:OFT}
C.~St. J.~A. Nash-Williams.
\newblock On well-quasi-ordering finite trees.
\newblock {\em Proc. Cambridge Philos. Soc.}, 59:833--835, 1963.

\bibitem[NW64]{Na:OWLS}
C.~St. J.~A. Nash-Williams.
\newblock On well-quasi-ordering lower sets of finite trees.
\newblock {\em Proc. Cambridge Philos. Soc.}, 60:369--384, 1964.

\bibitem[Ram29]{Ra:OAPO}
F.~P. Ramsey.
\newblock {On a problem of formal logic}.
\newblock {\em Proceedings London Mathematical Society (2)}, 30:264--286, 1929.

\bibitem[RS04]{RS:GM}
N.~Robertson and P.~D. Seymour.
\newblock Graph minors. {XX}. {W}agner's conjecture.
\newblock {\em J. Combin. Theory Ser. B}, 92(2):325--357, 2004.

\bibitem[Ru{\v{s}}15]{Ru:WICA}
N.~Ru{\v{s}}kuc.
\newblock Well quasi-order in combinatorics and algebra.
\newblock Talk at the {53rd Summer School on General Algebra and Ordered Sets
  in Srn{\'{i}}, Czech Republic}, September 2015.
\newblock Slides available at {\tt
  http://www.karlin.mff.cuni.cz/\~{}ssaos/2015/}.

\end{thebibliography}
\def\cprime{$'$}

\end{document}